\documentclass[11pt]{amsart}

\usepackage[margin=1.25in]{geometry}

\usepackage{graphicx}

\usepackage{wasysym}
\usepackage[misc]{ifsym}

\usepackage{amsmath, amssymb, amsthm, graphicx, enumerate, float, bbm, nicefrac, mathrsfs, mathtools, comment, color, soul}
\usepackage[colorlinks]{hyperref}
\hypersetup{citecolor=blue}

\makeatletter
\@namedef{subjclassname@2020}{%
  \textup{2020} Mathematics Subject Classification}
\makeatother

\newtheorem{theorem}{Theorem}[section]
\newtheorem{lemma}[theorem]{Lemma}
\newtheorem{corollary}[theorem]{Corollary}
\newtheorem{proposition}[theorem]{Proposition}
\newtheorem*{conjecture}{Conjecture}

\theoremstyle{definition}
\newtheorem{definition}[theorem]{Definition}

\theoremstyle{remark}
\newtheorem{remark}[theorem]{Remark}

\numberwithin{equation}{section}

\begin{document}

\allowdisplaybreaks

\title[Properties of reproducing kernel Hilbert spaces]{Properties of reproducing kernel Hilbert spaces of a group action}

\author[T. Blom]{Tyler Blom}
\address{Department of Mathematics, St. Norbert College, De Pere, Wisconsin}
\email{tyler.blom@snc.edu}

\author[S.A. Hokamp]{Samuel A. Hokamp}
\address{Department of Computer Studies and Mathematics, University of Dubuque, Dubuque, Iowa}
\email{shokamp@dbq.edu}

\author[A. Jimenez]{Alejandro Jimenez}
\address{Department of Mathematics and Statistics, University of Maine, Orono, Maine}
\email{alejandro.jimenez@maine.edu}

\author[J. Laubacher]{Jacob Laubacher}
\address{Department of Mathematics, St. Norbert College, De Pere, Wisconsin}
\email{jacob.laubacher@snc.edu}


\date{\today}

\keywords{Group actions, reproducing kernels, spaces of continuous functions.\\\indent\emph{Corresponding author.} Samuel A. Hokamp \Letter~\href{mailto:shokamp@dbq.edu}{shokamp@dbq.edu} \phone~563-589-3231.}

\thanks{The first and third authors received funding from \emph{The Collaborative} at St. Norbert College}

\begin{abstract}
In this paper, we investigate properties of a reproducing kernel Hilbert space of a group action. In particular, we introduce an equivalence relation on a compact Hausdorff space $X$, and consequently establish three equivalent definitions for when two elements are related. We also see how the equivalence classes of $X$ correspond to subgroups of the group acting transitively on $X$, which we aptly refer to as relation stabilizers.
\end{abstract}

\maketitle

\section{Introduction}

In their 1976 paper \cite{NR}, Nagel and Rudin establish that every closed unitarily-invariant space of continuous functions defined on the unit sphere of $\mathbb{C}^n$ is the closed direct sum of a certain collection of unitarily-invariant spaces. This collection is a unique subcollection of a larger, fixed collection of closed unitarily-invariant spaces that are all pairwise orthogonal (in the sense of $L^2$) and \textit{minimal} in the sense that none of them contains a proper closed unitarily-invariant subspace. In essence, Nagel and Rudin show that every closed unitarily-invariant space is constructed from the same collection of simple building blocks. Nagel and Rudin extend this result to closed unitarily-invariant spaces of $L^p$-functions on the unit sphere (for $1\leq p<\infty$), and in \cite{H}, the second author devises a suitable analogue for $L^\infty$-functions.

The techniques used in \cite{H} and \cite{NR} are largely independent of the particular environment of unitarily-invariant spaces of functions defined on the unit sphere of $\mathbb{C}^n$, and the authors have conjectured that a more general result for spaces of continuous functions that are invariant under a group action must exist. This conjecture is stated below:

\begin{conjecture}
Let $G$ be a compact group acting continuously and transitively on a compact Hausdorff space $X$. Then there exists a collection $\mathscr{G}$ of closed, pairwise orthogonal, and minimal spaces of continuous functions on $X$, each invariant under the group action, such that any closed space of continuous functions on $X$ that is invariant under the group action is the closed direct sum of a unique subcollection of $\mathscr{G}$.
\end{conjecture}

Existence of a collection $\mathscr{G}$ with the desired properties is a result of the Peter-Weyl Theorem from \cite{PW}. Thus, verification of the conjecture only requires showing that every closed space of continuous functions invariant under the group action can be written as the closed direct sum of a unique subcollection of $\mathscr{G}$. In \cite{H2}, the second author verifies the conjecture under a further restriction on the collection $\mathscr{G}$, and this result is extended to spaces of $L^\infty$-functions in \cite{H1}. The conjecture in its full generality is yet to be proven, but this paper takes strides toward establishing the conjecture in full.

A key result from \cite{H2} (Theorem 3.5, stated in Section \ref{prelims} as Theorem \ref{oldThm}) establishes that each element $H$ of $\mathscr{G}$ must be a reproducing kernel Hilbert space, and thus the focus of our paper turns to the study of what the authors have come to call \textit{reproducing kernel Hilbert spaces of a group action}. In particular, this paper establishes results that follow from the relationship between the kernel functions $K_x$ from $H$ and the group action itself. Section~\ref{main} develops these results, prominent among which are Lemma \ref{eqrel} and Theorem \ref{tfaemain}, which establish an equivalence relation on $X$ and consequent properties. This relation ties to a class of subgroups of $G$, which are studied in Section \ref{more}. It is the authors' belief that the material of these sections is key to the full verification of the above conjecture.

\section{Preliminaries}\label{prelims}

Throughout this paper, we will assume that all scalars are taken from $\mathbb{C}$ unless otherwise stated. Next, we fix $X$ to be a compact Hausdorff space, and $C(X)$ will denote the space of continuous complex-valued functions with domain $X$. Furthermore, we let $G$ be a compact group which acts continuously and transitively on $X$. Concerning the group action, we use the conventional shorthand of $\alpha x$ to signify the action of $\alpha$ in $G$ on the element $x$ from $X$. However, there are times when it is advantageous to be more explicit with the group action, in which case we can use the notation of the map $\varphi_\alpha:X\longrightarrow X$. Here, for any $x\in X$, we have that $\varphi_\alpha(x)=\alpha\cdot x=\alpha x$. We aim to keep this paper as self-contained as possible, but for more reading on group theory, we direct the interested reader to a foundational reference like \cite{DF}.

A space of complex functions defined on $X$ is \textit{$G$-invariant} if $f\circ\varphi_\alpha$ remains within the space for every $f$ in the space and every $\alpha\in G$. This definition appears in \cite{H2}, but it is only a generalization of particular cases (such as that of unitarily-invariant spaces from \cite{NR}). Naturally, $C(X)$ is our first example of a $G$-invariant space.

We let $\mu$ denote the unique regular Borel probability measure on $X$ that is invariant under the group action. Specifically, we have
\begin{equation}\label{intinv}
\int_Xf~d\mu=\int_Xf\circ\varphi_\alpha~d\mu
\end{equation}
for all $f\in C(X)$ and all $\alpha\in G$. Existence of such a measure is a result of Weil from \cite{W}. We then define an inner product $[\cdot,\cdot]$ on $C(X)$ in the natural way:
\begin{equation}\label{intinn}
[f,g]=\int_Xf\overline{g}~d\mu.
\end{equation}
As an easy consequence of \eqref{intinv} and \eqref{intinn}, one has the following expected results:

\begin{lemma}
For all $f,g\in C(X)$ and all $\alpha\in G$, we have that
\begin{enumerate}[(1)]
    \item $[f\circ\varphi_\alpha,g\circ\varphi_\alpha]=[f,g]$, and
    \item $[f\circ\varphi_\alpha,g]=[f,g\circ\varphi_{\alpha^{-1}}]$.
\end{enumerate}
\end{lemma}

A space $H$ of complex functions defined on $X$ is a \textit{reproducing kernel Hilbert space} (RKHS) provided that $H$ is a Hilbert space, and for each $x\in X$, the evaluation functional $E_x(f)=f(x)$, for all $f\in H$, is bounded. This is equivalent to requiring that there be a unique function $K_x\in H$ for each $x\in X$ such that $f(x)$ is recovered when the inner product of $f$ and $K_x$ is taken, for all $f\in H$. For more reading on general Hilbert space theory, we refer the reader to \cite{R}. Moreover, one can see \cite{A} or \cite{PR} for additional information on reproducing kernel Hilbert spaces.

A concept which appeared in \cite{H2} but was not formally defined at the time is that of a \textit{reproducing kernel Hilbert space of a group action}, by which we mean a RKHS of continuous functions defined on $X$, whose inner product is given by \eqref{intinn}. Theorem 3.5 of \cite{H2} (stated below), which is foundational to the results established in Sections \ref{main} and \ref{more}, is actually then a result pertaining to reproducing kernel Hilbert spaces of a group action. We have adjusted the statement of the theorem below to reflect better its relationship to this concept.

\begin{theorem}[Theorem 3.5 from \cite{H2}]\label{oldThm}
Let $H$ be a nontrivial, closed $G$-invariant subspace of $C(X)$. Further, if $H$ is closed in $L^2(\mu)$, then $H$ is a reproducing kernel Hilbert space. In particular, to each $x\in X$ corresponds a unique $K_x\in H$ such that
$$
f(x)=[f,K_x]
$$
for all $f\in H$. Additionally, the functions $K_x$ satisfy the following for all $x,y\in X$:
\begin{enumerate}[(1)]
    \item $K_x(y)=\overline{K_y(x)}$,
    \item $\displaystyle f=\int_Xf(x)K_x~d\mu(x)$ for all $f\in H$,
    \item\label{oldact} $K_{\varphi_\alpha(x)}=K_x\circ\varphi_{\alpha^{-1}}$,
    \item $K_x=K_x\circ\varphi_\alpha$ for all $\alpha\in G$ such that $\alpha x=x$, and
    \item\label{cpart} $K_x(x)=K_y(y)>0$.
\end{enumerate}
\end{theorem}

The particular conclusion of Theorem \ref{oldThm} is that a closed $G$-invariant space of continuous functions which is also closed under the inner product from \eqref{intinn} is a reproducing kernel Hilbert space of a group action. We also note that the aforementioned functions $K_x\in H$ that correspond uniquely to each $x\in X$ are referred to as \emph{kernels}. The main impetus of this paper is to uncover useful properties about these kernels (see Section \ref{main}).

Moreover, it is important to understand the gravity of Theorem \ref{oldThm}, and \eqref{cpart} in particular. Notice that since $K_x(x)=K_y(y)$ for all $x,y\in X$, we then have that $K_x(x)$ is a constant. We shall denote this as $c$. Furthermore, since $c=K_x(x)>0$, we can conclude that $c$ is a positive real number. This was shown in \cite{H2}, and from this we gather that none of the kernels $K_x$ are trivial, for if any of them were, we would have that $H$ would be trivial as well.

Finally, to round out this preliminary section, we recall the Cauchy-Schwarz Inequality, mainly to call attention to the addendum.

\begin{theorem}[Cauchy-Schwarz Inequality]
Suppose a vector space has an inner product $[\cdot,\cdot]$. For all vectors $x$ and $y$, we have that
$$
|[x,y]|\leq||x||_2\cdot||y||_2.
$$
This inequality is an equality if and only if $x=\lambda y$ for some scalar $\lambda$. 
\end{theorem}

\section{Properties of the Kernels}\label{main}

Throughout the paper, we let $H$ be a nontrivial reproducing kernel Hilbert space of a group action with kernels $K_x$, and we define $c=K_x(x)$ for any $x\in X$. As noted in Section~\ref{prelims}, we know $c$ is a positive real number. The goal of this section is to establish natural and convenient properties for kernels. We start by looking at how $c$ behaves with respect to these kernels $K_x$.

\begin{lemma}\label{lemmafirst}
Let $x,y\in X$. We have that
\begin{enumerate}[(i)]
    \item\label{lem1norm} $c=||K_x||_2^2$, and
    \item\label{lem1c} $|K_x(y)|\leq c$.
\end{enumerate}
\end{lemma}
\begin{proof}
Let $x,y\in X$. First we observe that
$$
c = K_x(x) = [K_x, K_x] = \sqrt{[K_x, K_x]}^2=||K_x||_2^2.
$$
Hence, the first statement is proved.

Next, using $\sqrt{c}=||K_x||_2$ for any $x\in X$ by way of $(i)$ above, we see that
$$
|K_x(y)|=|[K_x,K_y]|\le ||K_x||_2\cdot||K_y||_2=\sqrt{c} \cdot \sqrt{c}=c.
$$
Hence, the second statement is proved.
\end{proof}

We recall that two vectors $x$ and $y$ are orthogonal if their inner product is $0$. In notation, we write $x\perp y$.

\begin{lemma}\label{lemperp}
Let $f\in H$ and $x\in X$. We have that
\begin{enumerate}[(i)]
    \item $f\perp K_x$ if and only if $f(x)=0$, and
    \item $\Big(f-\frac{f(x)}{c}K_x\Big)\perp K_x$.
\end{enumerate}
\end{lemma}
\begin{proof}
Let $f\in H$ and $x\in X$. For the first statement, we start by supposing that $f \perp K_x$. Since $f \perp K_x$, then $[f,K_x]=0$. Note that $f(x)=[f,K_x]=0$. Thus $f(x)=0$.

For the converse, suppose that $f(x)=0$. Observe that $0=f(x)=[f,K_x]$. Since $0=[f,K_x]$, then $f \perp K_x$, as desired.

Next, for the second statement, we observe the following:
\begin{align*}
\left[f-\tfrac{f(x)}{c}K_x,K_x\right]&=[f,K_x]-\left[\tfrac{f(x)}{c}K_x, K_x\right]\\
&=[f,K_x]-\tfrac{f(x)}{c}\cdot [K_x,K_x]\\
&= [f,K_x]-\tfrac{f(x)}{c} \cdot K_x(x)\\
&=[f,K_x]-\tfrac{f(x)}{c} \cdot c \\
&=[f,K_x]-f(x) \\
&= f(x) - f(x) \\
&= 0.
\end{align*}
Since $\left[f-\frac{f(x)}{c}K_x,K_x\right]=0$, then $\left(f-\frac{f(x)}{c}K_x\right)\perp K_x$, which was what we wanted.
\end{proof}

It is always beneficial to know when two kernels are orthogonal, as we shall see below. But with respect to Lemma \ref{lemperp}, notice that this implies that for any two kernels $K_x$ and $K_y$, we have that $K_x\perp K_y$ if and only if $K_x(y)=0$. To wit, we can fully encode orthogonality in terms of evaluation.

\begin{proposition}
If $\{K_{x_i}\}_{i=1}^n$ is a finite set of kernels that is pairwise orthogonal, then the set is linearly independent.
\end{proposition}
\begin{proof}
It is sufficient to show this is true for $n=2$. We start by considering the finite set of nonzero kernels that is pairwise orthogonal $\{K_{x_i}\}_{i=1}^2$ with the homogeneous equation
$$
c_1K_{x_1}+c_2K_{x_2}=\Vec{0}.
$$
We then see that
$$
c_1K_{x_1}=-c_2K_{x_2}.
$$
Since we are supposing $K_{x_1}\perp K_{x_2}$, then $0=[K_{x_1},K_{x_2}]$. Therefore
$$
0=[K_{x_1},K_{x_2}]=c_1[K_{x_1},K_{x_2}]=[c_1K_{x_1},K_{x_2}]=[-c_2K_{x_2},K_{x_2}]=-c_2[K_{x_2},K_{x_2}]=-c_2\cdot c.
$$
Since $c>0$, it then follows that $-c_2=0$, meaning $c_2=0$. Also, since $c_1K_{x_1}=-c_2K_{x_2}=0$ with $K_{x_1}$ nonzero, we conclude that $c_1=0$. Thus, the only solution to the homogeneous equation is the trivial solution, meaning this set is linearly independent.
\end{proof}

\begin{proposition}
We have that $f=\frac{f(x)}{c}K_x$ if and only if $||f||_2^2=\frac{|f(x)|^2}{c}$.
\end{proposition}
\begin{proof}
Suppose $f=\frac{f(x)}{c}K_x$. Notice that
$$
||f||_2^2=[f,f]=\left[\tfrac{f(x)}{c}K_x,\tfrac{f(x)}{c}K_x\right]=\tfrac{f(x)}{c}\overline{\left(\tfrac{f(x)}{c}\right)}[K_x,K_x]=\left|\tfrac{f(x)}{c}\right|^2c=\tfrac{|f(x)|^2}{|c|^2}c=\tfrac{|f(x)|^2}{c},
$$
as desired.

To prove the converse, let $\pi_x$ be the orthogonal projection of $H$ onto $\text{span}\{K_x\}$, and $\pi_x^\perp$ the projection orthogonal to $\text{span}\{K_x\}$. Thus, for every $f\in H$, we have $\pi_x f=\lambda K_x$ for some scalar $\lambda$. Observe that $\pi_x f(x)=\lambda K_x(x)=\lambda c$. Since every $f\in H$ has an orthogonal decomposition $f=\pi_x f+\pi_x^\perp f$, and since $\pi_x^\perp f(x)=[\pi_x^\perp f,K_x]=0$, we get that $\pi_x f(x)=f(x)$. Thus, we get $\lambda=\frac{f(x)}{c}$, and in particular, we have $\pi_x f=\frac{f(x)}{c}K_x$ and $\pi_x^\perp f=f-\frac{f(x)}{c}K_x$ for all $f\in H$. A basic result of Hilbert space theory yields that $$||f||_2^2=||\pi_x f||_2^2+||\pi_x^\perp f||_2^2.$$ Thus, 
\begin{align*}
||f||_2^2&=||\pi_x f||_2^2+||\pi_x^\perp f||_2^2\\
&=||\tfrac{f(x)}{c}K_x||_2^2+||f-\tfrac{f(x)}{c}K_x||_2^2\\
&=\tfrac{|f(x)|^2}{c^2}||K_x||_2^2+||f-\tfrac{f(x)}{c}K_x||_2^2\\
&=\tfrac{|f(x)|^2}{c}+||f-\tfrac{f(x)}{c}K_x||_2^2.
\end{align*}
When $||f||_2^2=\frac{|f(x)|^2}{c}$, it must be the case that $||f-\frac{f(x)}{c}K_x||_2^2=0$, from which we get that $f-\frac{f(x)}{c}K_x=0$, and the result follows.
\end{proof}

\begin{proposition}\label{thmortho}
Let $H$ be finite-dimensional, and suppose $\{K_{x_i}\}_{i=1}^n$ is an orthogonal basis for $H$. We have that $f=\sum_{i=1}^n\frac{f(x_i)}{c}K_{x_i}$ for all $f\in H$.
\end{proposition}
\begin{proof}
It is sufficient to show this is true for $n=2$. We note that since $\{K_{x_i}\}_{i=1}^2$ is a basis for $H$, then $f=a_1K_{x_1}+a_2K_{x_2}$ for some scalars $a_1$ and $a_2$, and so $f-a_2K_{x_2}=a_1K_{x_1}$. Further, since the elements of $\{K_{x_i}\}_{i=1}^2$ are pairwise orthogonal, we know that $[K_{x_1},K_{x_2}]=0$. In particular, $0=a_1[K_{x_1},K_{x_2}]$, and so
\begin{align*}
0&=a_1[K_{x_1},K_{x_2}]\\
&=[a_1K_{x_1},K_{x_2}] \\
&=[f-a_2K_{x_2},K_{x_2}] \\
&=[f,K_{x_2}]+[-a_2K_{x_2},K_{x_2}]\\
&=f(x_2)-a_2[K_{x_2},K_{x_2}]\\
&=f(x_2)-a_2c.
\end{align*}
Since $0=f(x_2)-a_2c$, then we easily get that $a_2=\frac{f(x_2)}{c}$. Also, since $f-a_1K_{x_1}=a_2K_{x_2}$ and $[K_{x_2},K_{x_1}]=0$, we can follow similar steps to those above to show that $a_1=\frac{f(x_1)}{c}$. Thus, we have shown that $f=\frac{f(x_1)}{c}K_{x_1}+\frac{f(x_2)}{c}K_{x_2}=\sum_{i=1}^2\frac{f(x_i)}{c}K_{x_i}$.
\end{proof}

\begin{corollary}\label{corortho}
Let $H$ be finite-dimensional, and suppose $\{K_{x_i}\}_{i=1}^n$ is an orthogonal basis for $H$. We have that $\sum_{i=1}^n|K_{x_i}(x)|^2=c^2$ for all $x\in X$.
\end{corollary}
\begin{proof}
Again, it is sufficient to show this is true for $n=2$. Due to Proposition \ref{thmortho} with $f=K_x$, and recalling that $c=K_x(x)$, we have that 
\begin{align*}
K_x(x)&=\tfrac{K_x(x_1)}{c} \cdot K_{x_1}(x) +\tfrac{K_x(x_2)}{c} \cdot K_{x_2} (x) \\
c^2&=K_x(x_1) \cdot K_{x_1}(x) + K_x(x_2) \cdot K_{x_2} (x).
\end{align*}
Observe the following: 
\begin{align*}
|K_{x_1}(x)|^2 +|K_{x_2}(x)|^2 &=K_{x_1}(x) \cdot \overline{K_{x_1}(x) } + K_{x_2}(x) \cdot \overline{K_{x_2}(x)}\\
&= K_x(x_1) \cdot K_{x_1}(x) + K_x(x_2) \cdot K_{x_2}{(x)}\\
&=c^2.
\end{align*}
Hence, $|K_{x_1}(x)|^2+|K_{x_2}(x)|^2=c^2$, as desired.
\end{proof}

For the remainder of Section \ref{main}, we investigate the relationship for when two kernels are scalar multiples of each other. As we will see, especially in Theorem \ref{tfaemain} and Section \ref{more}, the consequences are as far-reaching as they are interesting.

\begin{definition}
Let $x,y\in X$. We say that $x\sim y$ if and only if $K_y=\lambda K_x$ for some scalar $\lambda\neq 0$.
\end{definition}

\begin{lemma}\label{eqrel}
The relation $\sim$ defines an equivalence relation.
\end{lemma}
\begin{proof}
First we observe that $K_x = \lambda K_x$ for all $x\in X$ with $\lambda=1$. Thus, $x \sim x$, and so our relation is reflexive.

Next, we let $x,y \in X$ and we suppose $x \sim y$, and therefore $K_y = \lambda K_x$ for some scalar $\lambda\neq0$. Then we see $K_x=\frac{1}{\lambda} K_y$, which implies $y \sim x$. Therefore, our relation is symmetric.

Finally, let $x,y,z \in X$ and suppose $x \sim y$ and $y \sim z$. By definition, $K_y = \lambda_1 K_x$ and $K_z = \lambda_2 K_y$ for some scalars $\lambda_1\neq0$ and $\lambda_2\neq0$. Taking $\lambda=\lambda_2\lambda_1$, where we notice that $\lambda\neq0$, we have
$$
\lambda K_x=(\lambda_2\lambda_1)K_x=\lambda_2(\lambda_1K_x)=\lambda_2K_y=K_z.
$$
Thus, $x \sim z$, and so our relation is transitive.

Hence, $\sim$ is an equivalence relation.
\end{proof}

\begin{theorem}\label{tfaemain}
Let $x,y\in X$. The following are equivalent:
\begin{enumerate}[(a)]
    \item $x\sim y$,
    \item $\beta x\sim \beta y$ for all $\beta \in G$,
    \item $|K_z(x)|=|K_z(y)|$ for all $z \in X$, and
    \item $|K_x(y)|=c$.
\end{enumerate}
\end{theorem}
\begin{proof}
Let $x,y\in X$. For $(a)\Rightarrow(b)$, we suppose $x\sim y$ and we let $\beta\in G$. Since $x\sim y$, then $K_y=\lambda K_x$ for some scalar $\lambda\neq 0$, and so $K_y\circ\varphi_{\beta^{-1}}=\lambda K_x\circ\varphi_{\beta^{-1}}$. Using Theorem \ref{oldThm}\eqref{oldact}, it follows that $K_{\beta y}=\lambda K_{\beta x}$. Thus, $\beta x\sim \beta y$, as desired.

For $(b)\Rightarrow(c)$, we let $z\in X$ and suppose $\beta x\sim \beta y$ for all $\beta \in G$. Taking $\beta=e$, we get that $x\sim y$, and so $K_y=\lambda K_x$ for some scalar $\lambda\neq0$. As consequence, we note that $K_y(x)=\lambda K_x(x)=\lambda c$, and so $\frac{K_y(x)}{c}=\lambda$. Additionally, we have that $c=K_y(y)=\lambda K_x(y)$, and so $\tfrac{c}{K_x(y)}=\lambda$, in which case
$$
\overline{\lambda}=\overline{\tfrac{c}{K_x(y)}}=\tfrac{\overline{c}}{\overline{K_x(y)}}=\tfrac{c}{K_y(x)}=\tfrac{1}{\lambda}.
$$
Therefore, we get that $|\lambda|^2=\lambda\overline{\lambda}=\lambda\frac{1}{\lambda}=1$, and hence, $|\overline{\lambda}|=|\lambda|=1$.
Now observe
$$
|K_z(y)|=|[K_z,K_y]|=|[K_z,\lambda K_x]|= |\overline{[\lambda K_x,K_z]}|=|\overline{\lambda}\overline{[K_x, K_z]}|=|\overline{\lambda}|\cdot|[K_z,K_x]|=|K_z(x)|,
$$
whence the result.

For $(c)\Rightarrow(d)$, we suppose $|K_z(x)|=|K_z(y)|$ for all $z \in X$. Taking $z=x$, we then have that $c=K_x(x)=|K_x(x)|=|K_x(y)|$.

For $(d)\Rightarrow(a)$, we suppose $|K_x(y)|=c$. Using $\sqrt{c}=||K_x||_2$ for any $x\in X$ by way of Lemma \ref{lemmafirst}\eqref{lem1norm}, observe that
$$
|[K_x, K_y]|=|K_x(y)|=c=\sqrt{c} \cdot \sqrt{c}= ||K_x||_2 \cdot ||K_y||_2.
$$
Since $|[K_x, K_y]| = ||K_x||_2 \cdot ||K_y||_2$, then by the addendum of the Cauchy-Schwarz Inequality, $K_x$ and $K_y$ must be scalar multiples of each other. Thus, $K_y=\lambda K_x$ for some scalar $\lambda\neq 0$, and as consequence $x\sim y$, which was what we wanted.
\end{proof}

\begin{remark}\label{remlambda}
Recall that whenever $x\sim y$, we have that $K_y=\lambda K_x$ for some scalar $\lambda\neq0$. Notice that within the proof of Theorem \ref{tfaemain}, we get $\lambda=\frac{K_y(x)}{c}$, and in particular $|\lambda|=1$.
\end{remark}

\begin{corollary}
For all $x,y\in X$, we have that $K_x(y)=c$ if and only if $K_y=K_x$.
\end{corollary}
\begin{proof}
Let $x,y \in X$. Suppose that $K_x(y)=c$. Since $c>0$, then $c=|c|=|K_x(y)|$. Since $|K_x(y)|=c$, then by Theorem \ref{tfaemain} we get that $K_y=\lambda K_x$ for some scalar $\lambda\neq0$. Next, by Remark \ref{remlambda}, we know that
$$
\lambda=\tfrac{K_y(x)}{c}=\tfrac{\overline{K_x(y)}}{c}=\tfrac{\overline{c}}{c}=\tfrac{c}{c}=1.
$$
Thus, $K_y=\lambda K_x=K_x$.

For the converse, suppose that $K_x = K_y$. Observe that $K_x(y)=K_y(y)=c$.
\end{proof}

One of the main goals of this paper is to investigate the equivalence classes in $X$ under the relation $\sim$ and how they correspond to subgroups of $G$. It turns out that working with the classic stabilizer is too restrictive, but the stabilizer under the relation (defined below) provides the correct framework to help us begin to answer that question.

\begin{definition}
Fix $x\in X$. We define the \textbf{relation stabilizer} of $x$ as the set
$$
\mathcal{E}(x)=\{\gamma\in G~|~\gamma x\sim x\}.
$$
It is easy to verify that $\mathcal{E}(x)$ is a subgroup of $G$.
\end{definition}

Since each relation stabilizer $\mathcal{E}(x)$ is a subgroup of $G$, then of course $\mathcal{E}(x)\neq\varnothing$ because $e\in\mathcal{E}(x)$. However, one wonders if there are nontrivial elements of $G$ that stabilize each $x\in X$ under the relation. The answer comes down to making sure that each equivalence class $[x]$ is nontrivial:

\begin{lemma}\label{lemclass}
Let $x\in X$. Then $|[x]|>1$ if and only if $|\mathcal{E}(x)|>1$.
\end{lemma}
\begin{proof}
Let $x\in X$. For the forward direction, we suppose $|[x]|>1$. This means there exists $y\in X$ with $x\neq y$ such that $x\sim y$. Since the action on $X$ is transitive, then there exists some $\alpha\in G$ such that $y=\alpha x$. Namely, since $x\neq y$, we know that $\alpha\neq e$. Therefore, we have that $\alpha x=y\sim x$, and so $\alpha\in\mathcal{E}(x)$. Hence, since $e,\alpha\in\mathcal{E}(x)$, we have that $|\mathcal{E}(x)|>1$.

For the converse, we suppose $|\mathcal{E}(x)|>1$. Therefore, we have that there exists some $\alpha\in\mathcal{E}(x)$ such that $\alpha\neq e$. Moreover, since $\alpha\in\mathcal{E}(x)$, we have that $\alpha x\sim x$. Set $y=\alpha x$. Since we know that $\alpha\neq e$, then we know that $\alpha x\neq x$. So $y\neq x$. However, we note that $y=\alpha x\sim x$, and so $y\sim x$. Thus, we have $x,y\in[x]$, and so $|[x]|>1$.
\end{proof}

\begin{proposition}\label{propclass}
Let $y\in X$. Then $|[y]|>1$ if and only if $|[x]|>1$ for all $x\in X$.
\end{proposition}
\begin{proof}
Let $y\in X$. For the forward direction, we suppose $|[y]|>1$. By Lemma \ref{lemclass}, we then know that $|\mathcal{E}(y)|>1$, in which case there exists some $\beta\in\mathcal{E}(y)$ such that $\beta\neq e$. In particular, we have that $\beta y\sim y$. Next, we let $x\in X$ be arbitrary. Since the action on $X$ is transitive, then there exists some $\alpha\in G$ such that $y=\alpha x$. Therefore $\alpha^{-1}y=x$. Consider $\alpha^{-1}\beta\alpha\in G$. Notice that
$$
(\alpha^{-1}\beta\alpha)x=\alpha^{-1}\beta(\alpha x)=\alpha^{-1}\beta y\sim\alpha^{-1}y=x,
$$
and so $(\alpha^{-1}\beta\alpha)x\sim x$. Thus $\alpha^{-1}\beta\alpha\in\mathcal{E}(x)$. For the sake of contradiction, suppose that $\alpha^{-1}\beta\alpha=e$. Therefore
$$
\beta=(\alpha\alpha^{-1})\beta(\alpha\alpha^{-1})=\alpha(\alpha^{-1}\beta\alpha)\alpha^{-1}=\alpha e\alpha^{-1}=\alpha\alpha^{-1}=e.
$$
This is our contradiction since we know that $\beta\neq e$, and so we must have that $\alpha^{-1}\beta\alpha\neq e$. Hence, we know that $e\in\mathcal{E}(x)$ and $\alpha^{-1}\beta\alpha\in\mathcal{E}(x)$, and so $|\mathcal{E}(x)|>1$. Again by Lemma \ref{lemclass}, we then get that $|[x]|>1$. Since $x$ was chosen arbitrarily, we have that $|[x]|>1$ for all $x\in X$.

For the converse, we suppose $|[x]|>1$ for all $x\in X$. Then since $y\in X$, we get that $|[y]|>1$, which was what we wanted.
\end{proof}

Notice that Lemma \ref{lemclass} and Proposition \ref{propclass} have heavy implications. If we have one element $x\in X$ such that its equivalence class is trivial, then every equivalence class in $X$ is also trivial. Moreover, if this were the case, every corresponding relation stabilizer subgroup $\mathcal{E}(x)$ would be trivial as well. This would be uninteresting. To avoid this, we need just a single equivalence class to be imposed as nontrivial, and therefore they all are.

We finish this section with how the relation on $X$ tells us information about the corresponding relation stabilizers in $G$.

\begin{theorem}\label{thmforward}
If $x\sim y$, then $\mathcal{E}(x)=\mathcal{E}(y)$.
\end{theorem}
\begin{proof}
Let $x,y\in X$. Suppose $x\sim y$. Let $\beta\in\mathcal{E}(x)$. Since $\beta\in\mathcal{E}(x)$, then $\beta x\sim x$. Furthermore, since $\beta\in G$ and $x\sim y$, then by Theorem \ref{tfaemain}, we have that $\beta x\sim\beta y$. Notice that
$$
\beta y\sim \beta x\sim x\sim y,
$$
and so $\beta\in\mathcal{E}(y)$. Hence $\mathcal{E}(x)\subseteq\mathcal{E}(y)$, and we get $\mathcal{E}(y)\subseteq\mathcal{E}(x)$ by symmetry. Thus, $\mathcal{E}(x)=\mathcal{E}(y)$.
\end{proof}

One of the main goals of Section \ref{more} is to look at the conditions needed in order to obtain the converse of Theorem \ref{thmforward}.

\section{Properties of the Relation Stabilizers}\label{more}

In order to more fully understand the converse of Theorem \ref{thmforward}, we need to first understand some implications for when one relation stabilizer subgroup is contained in another.

\begin{proposition}
Let $x\in X$ and $\alpha\in G$. If $\mathcal{E}(\alpha x)\leq\mathcal{E}(x)$, then for any $\beta\in\mathcal{E}(\alpha x)$, we have that
\begin{enumerate}[(i)]
    \item $\alpha\beta\alpha^{-1}\in\mathcal{E}(\alpha x)$,
    \item $\alpha\beta \alpha^{-1}\beta^{-1}\in\mathcal{E}(\alpha x)$, and
    \item $\beta^{-1}\alpha^{-1}\beta\alpha\in\mathcal{E}(x)$.
\end{enumerate}
\end{proposition}
\begin{proof}
Let $x\in X$ and $\alpha\in G$. Suppose $\mathcal{E}(\alpha x)\leq\mathcal{E}(x)$ and let $\beta\in\mathcal{E}(\alpha x)$. Notice that by supposition, we also get $\beta\in\mathcal{E}(x)$. Observe
$$
(\alpha\beta\alpha^{-1})(\alpha x)=\alpha(\beta x)\sim\alpha x,
$$
and hence $(i)$ is established.

Next, we can easily get $(ii)$ due to $\alpha\beta\alpha^{-1}\in\mathcal{E}(\alpha x)$ by above, and because we know $\beta^{-1}\in\mathcal{E}(\alpha x)$.

Finally, observing that $\beta^{-1}\in\mathcal{E}(x)$, we see that
$$
(\beta^{-1}\alpha^{-1}\beta\alpha)x=(\beta^{-1}\alpha^{-1})\beta(\alpha x)\sim (\beta^{-1}\alpha^{-1})\alpha x=\beta^{-1} x\sim x,
$$
which proves $(iii)$.
\end{proof}

What makes the above so interesting is that it shows both conjugation and commutators are inherited from a subgroup containment. This allows us to realize the form of some elements in the relation stabilizer. Furthermore, as we see below in Lemma \ref{lemeprop}, the action on the group also distributes nicely over the relation stabilizer.

\begin{lemma}\label{lemeprop}
Let $x\in X$ and $\beta\in G$. We have that
\begin{enumerate}[(i)]
    \item $\beta\in\mathcal{E}(x)$ if and only if $\beta\in\mathcal{E}(\beta x)$, and
    \item\label{sublemdist} $\mathcal{E}(\beta x)=\beta \mathcal{E}(x)\beta^{-1}$.
\end{enumerate}
\end{lemma}
\begin{proof}
Let $x\in X$ and $\beta\in G$. First, we suppose $\beta\in\mathcal{E}(x)$, and so $\beta x\sim x$. Applying Theorem \ref{tfaemain} yields $\beta(\beta x)\sim\beta x$. Thus $\beta\in\mathcal{E}(\beta x)$.
    
To show the converse of $(i)$, suppose $\beta\in\mathcal{E}(\beta x)$, and so $\beta(\beta x)\sim\beta x$. We again employ Theorem \ref{tfaemain} (with $\beta^{-1}$ this time) to get $\beta x\sim x$. Thus, $\beta\in\mathcal{E}(x)$.

For $(ii)$, we start with $\gamma\in\mathcal{E}(\beta x)$. Since $\gamma\in\mathcal{E}(\beta x)$, then $\gamma\beta x\sim\beta x$. This implies that $\beta^{-1}\gamma\beta x\sim x$, in which case $\beta^{-1}\gamma\beta\in\mathcal{E}(x)$. Observe that
$$
\gamma=e\gamma e=(\beta\beta^{-1})\gamma(\beta\beta^{-1})=\beta(\beta^{-1}\gamma\beta)\beta^{-1}\in\beta\mathcal{E}(x)\beta^{-1}.
$$
Hence, $\mathcal{E}(\beta x)\subseteq\beta\mathcal{E}(x)\beta^{-1}$.

Next, we let $\gamma\in\beta\mathcal{E}(x)\beta^{-1}$. Since $\gamma\in\beta\mathcal{E}(x)\beta^{-1}$, then $\gamma=\beta\delta\beta^{-1}$ for some $\delta\in\mathcal{E}(x)$, in which case $\delta x\sim x$. Notice that
$$
\gamma(\beta x)=\beta\delta\beta^{-1}(\beta x)=\beta\delta(\beta^{-1}\beta)x=\beta\delta e x=\beta\delta x\sim\beta x.
$$
Since $\gamma(\beta x)\sim\beta x$, then $\gamma\in\mathcal{E}(\beta x)$, and so $\beta\mathcal{E}(x)\beta^{-1}\subseteq\mathcal{E}(\beta x)$. The result follows.
\end{proof}

As we shall see below, along with the rest of Section \ref{more}, Lemma \ref{lemeprop}\eqref{sublemdist} in particular has far-reaching consequences.

\begin{proposition}
Let $x\in X$ and $\alpha\in G$. Then $\mathcal{E}(\alpha x)\leq\mathcal{E}(x)$ if and only if $\mathcal{E}(x)\leq\mathcal{E}(\alpha^{-1} x)$.
\end{proposition}
\begin{proof}
Let $x\in X$ and $\alpha\in G$. Suppose $\mathcal{E}(\alpha x)\leq\mathcal{E}(x)$. Using Lemma \ref{lemeprop}\eqref{sublemdist}, we see that
$$
\mathcal{E}(x)=e\mathcal{E}(x)e=\alpha^{-1}\alpha\mathcal{E}(x)\alpha^{-1}\alpha=\alpha^{-1}\mathcal{E}(\alpha x)\alpha\leq\alpha^{-1}\mathcal{E}(x)\alpha=\mathcal{E}(\alpha ^{-1} x).
$$

The converse follows in a similar manner.
\end{proof}

It turns out that the relation stabilizer may not necessarily be normal. This leads us in the direction of investigating the world in which it will always be normal: the normalizer. As we shall see, this will be the correct lens in which to look that allows us to obtain our goal of getting conditions needed for the converse of Theorem \ref{thmforward}, which we will state in Theorem \ref{tfaecon}. But first, as an analogue to Lemma \ref{lemeprop}, we have the following:

\begin{lemma}\label{Normdiste}
Let $x\in X$ and $\beta\in G$. We have that
\begin{enumerate}[(i)]
    \item $\beta\in N_G(\mathcal{E}(x))$ if and only if $\beta\in N_G(\mathcal{E}(\beta x))$, and
    \item \label{sublemN} $N_G(\mathcal{E}(\beta x))=\beta N_G(\mathcal{E}(x))\beta^{-1}$.
\end{enumerate}
\end{lemma}
\begin{proof}
Let $x\in X$ and $\beta\in G$. We start by supposing $\beta\in N_G(\mathcal{E}(x))$. Since $\beta\in N_G(\mathcal{E}(x))$, then $\beta\mathcal{E}(x)\beta^{-1}=\mathcal{E}(x)$. Then, using Lemma \ref{lemeprop}\eqref{sublemdist}, we observe that
$$
\mathcal{E}(\beta(\beta x))=\beta\mathcal{E}(\beta x)\beta^{-1}=\beta(\beta\mathcal{E}(x)\beta^{-1})\beta^{-1}=\beta\mathcal{E}(x)\beta^{-1}=\mathcal{E}(\beta x).
$$
Since $\mathcal{E}(\beta(\beta x))=\mathcal{E}(\beta x)$, then $\beta\in N_G(\mathcal{E}(\beta x))$.
    
For the converse, we suppose $\beta \in N_G(\mathcal{E}(\beta x))$. Since $\beta\in N_G(\mathcal{E}(\beta x))$, then $\mathcal{E}(\beta(\beta x))=\mathcal{E}(\beta x)$. We observe that
$$
\mathcal{E}(x)=\mathcal{E}(\beta^{-1}\beta x)=\beta^{-1}\mathcal{E}(\beta x)\beta=\beta^{-1}\mathcal{E}(\beta(\beta x))\beta=\beta^{-1}\beta\mathcal{E}(\beta x)\beta^{-1}\beta=\mathcal{E}(\beta x).
$$
Since $\mathcal{E}(x)=\mathcal{E}(\beta x)$, then $\beta\in N_G(\mathcal{E}(x))$, and so $(i)$ is established.

For $(ii)$, we first suppose that $\gamma\in N_G(\mathcal{E}(\beta x))$, which means that $\mathcal{E}(\gamma\beta x)=\mathcal{E}(\beta x)$. Next, consider $\delta=\beta^{-1}\gamma\beta$. Notice that
$$
\beta\delta\beta^{-1}=\beta(\beta^{-1}\gamma\beta)\beta^{-1}=(\beta\beta^{-1})\gamma(\beta\beta^{-1})=\gamma.
$$
Hence, $\gamma=\beta\delta\beta^{-1}$. We observe that
$$
\mathcal{E}(\delta x)=\mathcal{E}((\beta^{-1}\gamma\beta) x)=\beta^{-1}\mathcal{E}(\gamma\beta x)\beta=\beta^{-1}\mathcal{E}(\beta x)\beta=\mathcal{E}(\beta^{-1}\beta x)=\mathcal{E}(x),
$$
and so $\delta\in N_G(\mathcal{E}(x))$. Since $\gamma=\beta\delta\beta^{-1}$ such that $\mathcal{E}(\delta x)=\mathcal{E}(x)$, we see that $\gamma\in\beta N_G(\mathcal{E}(x))\beta^{-1}$. Thus, we have shown that $N_G(\mathcal{E}(\beta x))\subseteq\beta N_G(\mathcal{E}(x))\beta^{-1}$.
    
Next, suppose $\gamma\in\beta N_G(\mathcal{E}(x))\beta^{-1}$, which means that $\gamma = \beta\delta\beta^{-1}$ for some $\delta\in N_G(\mathcal{E}(x))$. Since $\delta\in N_G(\mathcal{E}(x))$, we have that $\mathcal{E}(\delta x)=\mathcal{E}(x)$. We observe that
$$
\mathcal{E}(\gamma\beta x)=\mathcal{E}((\beta\delta\beta^{-1})\beta x)=\mathcal{E}(\beta\delta x)=\beta\mathcal{E}(\delta x)\beta^{-1}=\beta\mathcal{E}(x)\beta^{-1}=\mathcal{E}(\beta x).
$$
Thus, since $\mathcal{E}(\gamma\beta x)=\mathcal{E}(\beta x)$, it follows that $\gamma\in N_G(\mathcal{E}(\beta x))$, and so $\beta N_G(\mathcal{E}(x))\beta^{-1}\subseteq N_G(\mathcal{E}(\beta x))$, whence the result.
\end{proof}

The transitivity of the group action is a rather powerful but not unreasonable condition. It is on full display in what follows.

\begin{proposition}\label{conjugateGroups}
Let $x,y\in X$. We have that
\begin{enumerate}[(i)]
    \item $\mathcal{E}(x)$ and $\mathcal{E}(y)$ are conjugate subgroups of $G$, and
    \item $N_G(\mathcal{E}(x))$ and $N_G(\mathcal{E}(y))$ are conjugate subgroups of $G$.
\end{enumerate}
\end{proposition}
\begin{proof}
Let $x,y\in X$. Since the action on $X$ is transitive, then there exists some $\alpha\in G$ such that $y=\alpha x$. For $(i)$, we use Lemma \ref{lemeprop}\eqref{sublemdist} to observe that
$$
\alpha \mathcal{E}(x)\alpha^{-1}=\mathcal{E}(\alpha x)=\mathcal{E}(y).
$$
Thus, we see that $\mathcal{E}(x)$ and $\mathcal{E}(y)$ are conjugate subgroups of $G$.

Similarly for $(ii)$, we use Lemma \ref{Normdiste}\eqref{sublemN} and find that
$$
\alpha N_G(\mathcal{E}(x))\alpha^{-1}=N_G(\mathcal{E}(\alpha x))=N_G(\mathcal{E}(y)).
$$
Hence, we conclude that $N_G(\mathcal{E}(x))$ and $N_G(\mathcal{E}(y))$ are also conjugate subgroups of $G$.
\end{proof}

\begin{corollary}
Let $x,y\in X$. We have that
\begin{enumerate}[(i)]
    \item $\mathcal{E}(x)\cong\mathcal{E}(y)$, and
    \item $N_G(\mathcal{E}(x))\cong N_G(\mathcal{E}(y))$.
\end{enumerate}
\end{corollary}
\begin{proof}
Let $x,y\in X$. By Proposition \ref{conjugateGroups}, we know that $\mathcal{E}(x)$ and $\mathcal{E}(y)$ are conjugate and $N_G(\mathcal{E}(x))$ and $N_G(\mathcal{E}(y))$ are conjugate. Since conjugate subgroups are isomorphic, our results follow.
\end{proof}

Having the relation stabilizers all be isomorphic, along with their corresponding normalizers, is powerful. But we need to once again look at the more rigid requirement of equality and how this can tie together with the normalizer. Lemma \ref{lemnormelement} addresses that question.

\begin{lemma}\label{lemnormelement}
Let $x\in X$ and $\alpha\in G$. Then $\mathcal{E}(\alpha x)=\mathcal{E}(x)$ if and only if $\alpha\in N_G(\mathcal{E}(x))$.
\end{lemma}
\begin{proof}
Let $x\in X$ and $\alpha\in G$. For the forward direction, we suppose $\mathcal{E}(\alpha x)=\mathcal{E}(x)$. Using Lemma \ref{lemeprop}\eqref{sublemdist}, we see that $\mathcal{E}(x)=\mathcal{E}(\alpha x)=\alpha\mathcal{E}(x)\alpha^{-1}$. Since $\mathcal{E}(x)=\alpha\mathcal{E}(x)\alpha^{-1}$, then $\alpha\in N_G(\mathcal{E}(x))$.

For the converse, we suppose $\alpha\in N_G(\mathcal{E}(x))$. Since $\alpha\in N_G(\mathcal{E}(x))$, then we have that $\mathcal{E}(x)=\alpha\mathcal{E}(x)\alpha^{-1}$. Again using Lemma \ref{lemeprop}\eqref{sublemdist}, we have $\mathcal{E}(\alpha x)=\alpha\mathcal{E}(x)\alpha^{-1}=\mathcal{E}(x)$.
\end{proof}

We recall that a subgroup is said to be self-normalizing if it is equal to its own normalizer. Finally, we can state Theorem \ref{tfaecon} below, housing the converse of Theorem \ref{thmforward}.

\begin{theorem}\label{tfaecon}
Let $x,y\in X$. The following are equivalent:
\begin{enumerate}[(a)]
    \item If $\mathcal{E}(x)=\mathcal{E}(y)$, then $x\sim y$,
    \item $\mathcal{E}(z)$ is self-normalizing for all $z\in X$, and
    \item $\mathcal{E}(x)$ is self-normalizing.
\end{enumerate}
\end{theorem}
\begin{proof}
Let $x,y\in X$. For $(a)\Rightarrow(b)$, we suppose $\mathcal{E}(x)=\mathcal{E}(y)$ implies $x\sim y$. Next, we let $z\in X$ be arbitrary. In order to show that $\mathcal{E}(z)$ is self-normalizing, it is sufficient to prove $N_G(\mathcal{E}(z))\leq\mathcal{E}(z)$. Let $\beta\in N_G(\mathcal{E}(z))$, and so $\mathcal{E}(\beta z)=\mathcal{E}(z)$ by Lemma \ref{lemnormelement}. Since $\mathcal{E}(\beta z)=\mathcal{E}(z)$, then by supposition we have that $\beta z\sim z$, in which case $\beta\in\mathcal{E}(z)$. Hence, $\mathcal{E}(z)$ is self-normalizing.

Next, for $(b)\Rightarrow(c)$, we suppose $\mathcal{E}(z)$ is self-normalizing for all $z\in X$. Taking $z=x$ gives our desired result.

Finally, for $(c)\Rightarrow(a)$, we suppose $\mathcal{E}(x)$ is self-normalizing. To prove $(a)$, we further suppose $\mathcal{E}(x)=\mathcal{E}(y)$. Since the action on $X$ is transitive, then there exists some $\alpha\in G$ such that $y=\alpha x$. By Lemma \ref{lemnormelement}, we have that $\alpha\in N_G(\mathcal{E}(x))$. Since $\mathcal{E}(x)$ is self-normalizing, then $\mathcal{E}(x)=N_G(\mathcal{E}(x))$. Therefore, $\alpha\in N_G(\mathcal{E}(x))=\mathcal{E}(x)$ and so $y=\alpha x\sim x$.
\end{proof}

Notice how Theorem \ref{tfaecon} gives us both a necessary and sufficient condition for what is needed to obtain the converse of Theorem \ref{thmforward}: one relation stabilizer has to be self-normalizing. In fact, if one is, then they all are.

\section{Future Work}

Due to Lemma \ref{lemmafirst}\eqref{lem1c}, it is clear that $0\leq|K_x(y)|\leq c$ for all $x,y\in X$. We know the consequences for when $|K_x(y)|$ attains the poles. In particular, whenever $|K_x(y)|=c$, we recall that Theorem \ref{tfaemain} implies $x\sim y$, among others. Indeed, we conjecture that requiring $|K_x(y)|>0$ is sufficient to obtain $x\sim y$. Moreover, whenever $|K_x(y)|=0$, recall that Lemma \ref{lemperp} implies $K_x\perp K_y$. As another open question, we suspect that $K_x\perp K_y$ if and only if $\mathcal{E}(x)\cap\mathcal{E}(y)=\{e\}$, which would mimic the goal of Section \ref{more} and provide information on how the kernels relate to the relation stabilizer subgroups for this specific scenario.

Furthermore, note that the results of Proposition \ref{thmortho} and Corollary \ref{corortho} both require a finite set of kernels acting as an orthogonal basis for $H$. One of our unsolved questions here is if such an orthogonal basis can always be found. Explicitly, if $H$ is a finite-dimensional $G$-invariant space and $x\in X$, then does there exist a finite set of kernels $\{K_{x_i}\}_{i=1}^n$ that forms an orthogonal basis for $H$ such that $x_1=x$?

Finally, we turn towards the idea of two reproducing kernel Hilbert spaces $H_1$ and $H_2$ and a linear map $T$ between them. It is natural to wonder what properties are preserved by $T$. Specifically, for $x,y\in X$ and their respective kernels $K_x$ and $K_y$ in $H_1$, does $(TK_x)(x)=(TK_y)(y)$ in $H_2$? If $(TK_x)(x)=0$, does this imply that $T=0$? Does $T$ map kernels to kernels? In other words, for any kernel $K_x$ in $H_1$, is $TK_x$ a kernel in $H_2$? We plan to investigate such matters in a follow-up paper.


\end{document}